\documentclass{amsart}
\usepackage{amsmath,amsthm,amsfonts,amssymb,latexsym,mathrsfs,graphicx}

\usepackage{hyperref}

\usepackage{enumerate}
\usepackage[shortlabels]{enumitem}
\usepackage{color}
\usepackage {tikz}
\usetikzlibrary {positioning}
\definecolor {processred}{cmyk}{0,0.96,0,0}

\newtheorem{theorem}{Theorem}[section]
\newtheorem{corollary}[theorem]{Corollary}
\newtheorem{lemma}[theorem]{Lemma}
\newtheorem{remark}[theorem]{Remark}

\theoremstyle{definition}

\newtheorem{example}[theorem]{Example}
\newtheorem{nota}[theorem]{Notation}

\tikzstyle{vertex}=[circle, draw, inner sep=0pt, minimum size=6pt] 
\newcommand{\vertex}{\node[vertex]}

\makeatletter
\newcommand*{\rom}[1]{\expandafter\@slowromancap\romannumeral #1@}
\makeatother
\begin{document}
	
\title{regular sets in Cayley  sum graphs}

\author[F. Seiedali]{Fateme sadat seiedali$^1$}
\author[B. Khosravi]{Behrooz Khosravi$^1$}
\author[Z. Akhlaghi]{Zeinab Akhlaghi$^{1,2}$}

\address{$^{1}$ Faculty of Mathematics and Computer Science, Amirkabir University of Technology (Tehran Polytechnic), 15914 Tehran, Iran.}
\address{$^{2}$ School of Mathematics,
	Institute for Research in Fundamental Science (IPM)
	P.O. Box:19395-5746, Tehran, Iran. }
\email{\newline \text{(F. Seiedali) }fseiedali@gmail.com \newline \text{(Z. Akhlaghi) }z\_akhlaghi@aut.ac.ir  \newline \text{(B. Khosravi) }khosravibbb@yahoo.com}

%\MSC[MSC 2010]{05C25 · 05C69 · 94B99}

\begin{abstract}
	
A subset $C$ of the vertex set of a graph $\Gamma$ is said to be $(\alpha,\beta)$-regular if  $C$ induces an $\alpha$-regular subgraph and every vertex outside $C$ is adjacent to exactly $\beta$ vertices in $C$. In particular, if $C$ is an $(\alpha,\beta)$-regular set in some Cayley sum  graph of a finite group $G$ with connection set $S$, then $C$ is called an $(\alpha,\beta)$-regular set of $G$ and a $(0,1)$-regular set is called a perfect code of $G$. By Sq$(G)$ and NSq$(G)$ we mean the set of all square elements and non-square elements of $G$. As one of the main results in this note, we show that a subgroup $H$ of a finite abelian group $G$ is an $(\alpha,\beta)$-regular set of $G$, for each $0\leq \alpha \leq |$NSq$(G)\cap H|$ and $0\leq \beta \leq \mathcal{L}(H)$, where $\mathcal{L}(H)=|H|$, if Sq$(G) \subseteq  H$ and $\mathcal{L}(H)=|$NSq$(G)\cap H|$, otherwise. As a consequence of our result we give a very brief proof for the main results in \cite{mama, ma}. Also, we consider the dihedral group $G=D_{2n} $ and for each subgroup $H $ of $G$, by giving an  appropriate  connection set $S$, we determine each  possibility for $(\alpha, \beta)$, where $H$ is an  $(\alpha,\beta)$-regular set of $G$.

\end{abstract}
\thanks{The third  author is supported by a Grant from IPM (no. 1403200013)}

%\maketitle

\maketitle

\linespread{1.5}

\section{Introduction}

Let   $C$ be a subset of   ${\mathbb F}_q^n$,  where ${\mathbb F}_q$ is the finite field of order $q$.  Then we call $C$ a code of lenght $n$ over   ${\mathbb F}_q$ and the elements of $C$ are called codwords.    For two vectors $x, y\in {\mathbb F}_q^n$,   their
Hamming distance
$d(x, y) $ is the number of
coordinates in which
they differ and  for $x\in {\mathbb F}_q^n$,  we define $d(x,C)=\min\{d(x,y)\ :\ y\in C\}$.   A code $ C$ over  ${\mathbb F}_q$ is called $q$-ary $\rho$-covering code  if for every vector $ y$ there is a codeword 
$x\in C$  such that the Hamming distance  $d(x,y)\leq \rho$.  The covering radius of a code $C$ is the smallest $\rho$ such that $C$ is $\rho$-covering. Let $x\in {\mathbb F}_q^n$.  Denote by $B_{x,i}$, the number of codewords of distance $i$ from $x$.   A code $C$ with covering radius $\rho$ is  called $t$-regular ($0\leq t\leq \rho$) if for all $i=0,\dots,\rho$,  $B_{x,i}$ depends only on $i$ and $d(x,C)$, for all $x$ such that  $d(x, C)\leq t$ (see for instance  \cite{del, goe}).  A code $C$ is completely regular if it is $\rho$-regular (see \cite{del}).   We refer to  \cite{davar}, for a comprehensive survey on completely regular codes. 
In $1$-regular code, we face two important parameters $\alpha=B_{x,1}$, for $x\in C$ and $\beta=B_{x, 1}$, for $x\in {\mathbb F}_q^n\setminus C$ (we have two other  trivial parametrs  $B_{x,0}=1$ for $x\in C$ and $B_{x,0}=0$ for $x\not\in C$, which can be ignored).  It means that for  $x\in C$, there exist exactly  $\alpha$ elements in $C$  of  distance 1  from $x$ and for all   $x\in {\mathbb F}_q^n\setminus C$, there exist exactly $\beta$ elements in $C$ of distance $1$ from $x$.  In particular, a $1$-regular code with $B_{x,1}=0$, for $x\in C$ and $B_{x, 1}=0$, for $x\in {\mathbb F}_q^n\setminus C$, is called a perfect code. 
% Completly regular codes are known for  Hamming, Golay,
%Preparata, some BCH codes with $d = 5$ and some Hadamard codes (see \cite{davar}).
 The
combinatorial properties of completely regular codes allow to establish different relations with other combinatorial structures such as distance-regular graphs, association schemes and designs. To extend the idea of regular code in other combinatorial structure, many authors pay attention
 perfect codes and regular sets in graphs. 

In this paper, all groups are finite and all graphs are simple. Let $\Gamma=({\bf V}(\Gamma), {\bf E}(\Gamma))$ be a simple graph, where ${\bf V}(\Gamma)$ and ${\bf E}(\Gamma)$ are the sets of its vertices and the set of its edges, respectively. For non-negative integers $\alpha,\beta$, a subset $C$ of ${\bf V}(\Gamma)$ is called an {\it$(\alpha,\beta) $-regular set} in $\Gamma$, if every vertex    of $C$ is adjacent to exactly $\alpha$ vertices of $C$ and every vertex of  ${\bf V}(\Gamma)\setminus C$ is a neighbor to exactly $\beta$ vertices of $C$ (see \cite{R1}). A $(0,1)$-regular set is called a {\it perfect code}.  Clearly,   the definition of regular set in a graph arises from the definition of the  $1$-regular code, by replacing the hamming distance of two vectors by the distance of the vertices of the graph, as the metric. 

For a graph $\Gamma=({\bf V}(\Gamma), {\bf E}(\Gamma))$, a partition of ${\bf V}(\Gamma)$ with cells $\mathcal{V}=\{ V_1,\dots,V_k\}$ is called an {\it equitable $k$-partition}, when each cell induces a regular subgraph and  any vertex of $V_i$ is adjacent to $b_{ij}$ vertices of $V_j$, \cite[Section 9.3]{godsil}. The quotient matrix of the partition $\mathcal{V}$ is defined as $M=(b_{ij})$. We note that, if the row sums of a $k \times k$ matrix $A$ is equal to a fixed number, say $r$, then $r$ is an eigenvalue of $A$ \cite[Theorem 9.3.3]{godsil}. Therefore, if $\Gamma$ is a connected $r$-regular graph, then $r$ is a simple eigenvalue of the quotient matrix $M$ of $\Gamma$.  
If $\Gamma$ is a $r$-regular graph, $\mathcal{V}$ an equitable partition of $V(\Gamma)$ and $M$ is its quotient matrix such that all eigenvalues of $M$ except $r$ are equal to $\mu$, then $\mathcal{V}$ is called $\mu$-equitable.  
We note that a non-trivial coarsening of a $\mu$-equitable partition is $\mu$-equitable (see \cite[ Corollary 2.3]{bcg}). So, the study of equitable partition with exactly two parts is important.

Moreover, an $(\alpha,\beta)$-regular set in a $r$-regular graph
$\Gamma$ is exactly a completely regular code $C$ in $\Gamma$ (see, for example, \cite{n}) such that the
corresponding distance partition has exactly two parts, namely $\{C, {\bf V}(\Gamma) \setminus  C\}$. 
 An equitable $2$-partition  is also called perfect 2-coloring \cite{2coloring}. The notion of perfect coloring is a common research subject in coding theory \cite{hamming2coloring1, hamming2coloring2}. 
 
 Let $G$ be a group with identity element $e$ and  $S$ be an inverse-closed subset  of $G$,   i.e.  $S^{-1} = \{s^{-1} : s \in S \} = S$,  where $e \notin S$. The {\it Cayley graph  } Cay$(G, S)$ of $G$ with respect to the connection set $S$ is defined to be the graph with vertex set $G$ such that two elements $x$, $y$ are adjacent if $yx^{-1} \in S$, see \cite{ taeri, DeVos}.
 
An element $ x $ of $ G $ is called a {\it square} if $ x = y^{2}$, for some element $ y \in G $. A subset of $ G $ whose elements are not square, is called {\it square-free}. A subset $ S $ of $ G $ is called a {\it normal} subset if for every $g \in G $, $ g^{-1} S g = \{ g^{ -1 }sg : s \in S \} = S$. 
Let $S$ be a normal square-free subset of $G$. The {\it Cayley sum graph}  CayS$(G,S)$ of $ G $ with respect to the connection set S is a simple graph with vertex set $ G $ and two vertices $x$ and $y$ are adjacent if $ x y \in S $. Since $ S $ is a normal square-free subset of $ G $, $x y \in S$ if and only if $ yx \in S $, hence CayS$(G, S) $ is an undirected graph without loops. It is easy to see that CayS$(G, S)$ is $|S|$-regular. The Cayley sum graph is first defined for abelian groups (see \cite{chung})  and then it is generalized to any arbitrary group in \cite{ma}.  
 The perfect code of  the cayley graphs of finite groups are widely  studied by some authors (see for instance \cite{0,paper, on}).   We note that in \cite{regularset}, a perfect code of a Cayley graph of $G$, briefly is called a perfect code of $G$. In \cite{mama, ma}, a subset $C$ of $G$ is called a {\it perfect code} of G if there exists a Cayley sum graph of $G$ which admits $C$ as a perfect code and in particular, if a subgroup $H$ of $G$ is a perfect code of $G$, then $H$ is called a {\it subgroup perfect code} of $G$. In the sequel of this paper, by a perfect code of $G$ we mean a perfect code of a Cayley sum graph of $G$ with respect to a subset $S$.

In \cite{mama}, Ma et al. provided necessary and sufficient conditions for a non-trivial subgroup of an abelian group to be a subgroup perfect code. They determine whether a subgroup of an abelian group is a perfect code according to its Sylow 2-subgroup. Moreover, they specified all subgroup perfect codes of a cyclic group, a dihedral group, and a generalized quaternion group. Also in \cite{ma}, they gave a shortened proof for classifying all perfect codes of abelian groups.

In  \cite{regularset},  it is proved that  a normal subgroup $H$ is a perfect code of a  Cayley graph of $G$ if and only if it is an $(\alpha,\beta)$-regular set in  a Cayley graph of $G$,  for each $0 \leq \alpha \leq |H|-1$ and $0 \leq \beta \leq |H|$, where gcd$(2, |H|-1)$ divides $\alpha$.

In \cite{note}, for a normal subgroup $H$ of $G$, it is proved that for   $0 \leq \alpha \leq |H|-1$  and $0 \leq \gamma\leq |H|/2$ such that gcd$(2, |H|-1)$ divides $\alpha$, $H$ is an $(\alpha,2\gamma)$-regular set in a Cayley  graph of $G$. 

 It is natural to ask about regular sets in  other graphs, specially  Cayley sum graphs. In the sequel of this paper,  by an  $(\alpha, \beta)$-regular set of $G$ we mean an $(\alpha, \beta)$-regular set in a Cayley sum graph of  $G$ and we study the following problem for   abelian groups and  dihedral groups  $G=D_{2n} $, for $n\geq 3$:

{\bf Problem}: Describe all  subgroups  $H$ of $G$ and parameters $\alpha$ and $\beta$  such that $H$ is an $(\alpha,\beta)$-regular  set of $G$. 

 To answer to  the above problem, for each subgroup $H$ of $G$  and $(\alpha, \beta)$ such that $H$ is an $(\alpha,\beta)$-regular set of $G$, we show that there exists  a connection  set  $S$ of  size $\alpha +\beta((|G|/|H|)-1)$ (Corollary \ref{gir})  such that $H$ is an  $(\alpha,\beta)$-regular set in   ${\rm CayS}(G,S)$.  We note that this graph has valency $|S|$ and it is worth mentioning  that  the connection set $S$ is not unique necessarily,  and sometimes the number of them can be  large (for example in the abelian groups). 

Let $G$ be an abelian group and $H$ a subgroup of $G$.  In the second section, we  determine all the possibilities for  $ \alpha$ and $\beta$, such that $H$ is  an $(\alpha,\beta)$-regular  set of  $G$.  As a consequence we give a shorter proof for the main results in \cite{ mama, ma}.

 In the third  section, we study  this problem  for  each  subgroup $H$ of the dihedral group $G$. In the proofs of the theorems in this section,  for each subgroup $H$ of $G$  and $(\alpha, \beta)$ such that $H$ is an $(\alpha,\beta)$-regular set of $G$, we introduce   a connection  set  $S$. 

 %we study the dihedral group $G=D_{2n} $ and for each subgroup $H$ of $G$ we determine all possibilities for $(\alpha, \beta)$  where  $H$ is $(\alpha,\beta)$-regular.  

Obviously, $G$ is a perfect code in the empty Cayley sum graph  CayS$(G, \emptyset ) $. As a result, any group is a subgroup perfect code of itself.  Similarly,  any  subgroup $H$ is  a $(0,0)$-regular set in ${\rm CayS}(G, \emptyset)$.  As this case is trivial,  in the rest of the paper  we do not consider the case when   $(\alpha, \beta)=(0,0)$.   It is easy to see that every non-trivial element of $G$ is non-square if and only if $G$ is an elementary abelian $2$-group. So a Cayley sum graph CayS$(G, S) $ is complete if and only if G is an elementary abelian $2$-group  and $ S = G \setminus \{ e \} $. This also means that the trivial subgroup $ \{e \} $ of $G$ is a subgroup perfect code if and only if $G$ is an elementary abelian $2$-group.

 Let $G$ be a finite group. Throughout  the paper, the  identity
 element  of $G$ is denoted by $e$.  The  number of elements of $G$ is denoted by $|G|$ and is called the  {\it order} of  $G$. The order of an element $g\in G$, denoted by $o(g)$,   is the smallest natural number $n$ such that $g^n=e$.  Let $H$ be a subgroup of $G$ and $a\in H$, then  $Ha=\{ha \ :\ h\in H\}$ is called  a  right coset of $H$ in $G$. The {\it index} of a subgroup $H$ in a group $G$ is the number of  distinct  right cosets of $H$ in $G$, which is denoted by  $[G:H]$.  A {\it right transversal} of $H$ in $G$ is a subset of G which contains exactly one element from each  right coset of H.  If the order of $H$ is the largest odd divisor of the order of $G$, then $H$ is called a {\it Hall $2^{\prime} $-subgroup} of $G$. 
 Let $|G|=p^{n}m$ where $p\nmid m$, then a subgroup  of $G$ of order $p^n$ is called  a {\it Sylow $p$-subgroup} of $G$.  
 An {\it elementary abelian} $p${\it-group} is an abelian group in which every non-trivial element has order $p$, where $p$ is a prime. The {\it direct product} of two groups $G$ and $H$ is  $G \times H=\{(g,h)\ :\  g\in G, \  h\in H\}$,  with the  group operation given by $(g_1,h_1). (g_2 , h_2)=(g_1 g_2 , h_1 h_2)$,  where the coordinate-wise operation are the operations in $G$ and $H$. Two elements $a$ and $b$ of $G$ are {\it conjugate} if there is an element $g \in G$ such that $b=gag^{-1}$. The {\it conjugacy class} of $a\in G$ is  the set of  all conjugates of $a$ in $G$, which is  denote by $a^G$. Let $X\subseteq G$ be a non-empty  subset of $G$, then  the smallest subgroup of $G$, containing $X$ is denoted by $\langle X\rangle$.  
 
 By Sq$(G)$ and NSq$(G)$ we mean the set of all square and non-square elements of $G$.   If  $G$ is an abelian group, then we  denote by ${\bf E}_2(G)$,  the greatest direct factor of $G$ isomorphic to an elementary abelian $2$-group. By $C_n$ we mean  the cyclic group of order $n$.

 \smallskip

\linespread{1.5}
\section{REGULAR SETS IN ABELIAN GROUPS }

\begin{lemma}\label{2.1}
	\cite[Lemma 2.1]{ma} Let $ S $  and $ H $ be a normal square-free subset and a subgroup of a finite group $G$, respectively.
	The following are equivalent:
	\begin{enumerate}
	
	\item H is a perfect code of  {\rm CayS}$(G, S) $.
	
\item $ S \cup \{ e \} $ is a right transversal of $ H $ in $
	G$.
	
	\item $ [G : H] = \left | S \right | + 1 $ and $ \left( S \cup SS^{-1} \right)\cap H = \{ e \}$. 
	\end{enumerate}
\end{lemma}

\begin{lemma}\cite[Theorem 3.1]{ma}\label{3.1mama}
Let $G$ be an abelian group with non-trivial Sylow $2$-subgroup $P$, and let $H$ be a non-trivial subgroup of $G$. Then $H$ is a subgroup perfect code of $G$ if and only if one of the following occurs:
	\begin{enumerate}
\item $P \subseteq  H$;
\item $[G : H]=|P|$ and $P$ is elementary abelian;
\item $P\cap H$ is a non-trivial subgroup perfect code of $P$,  and either $[G : H]$ is a power of $2$ or $ P\cap H$ has a non-square element in $G$.
	\end{enumerate}
\end{lemma}
\begin{lemma} \cite[Lemma 3.1]{mama} \label{3.1ma}
Let $A=C_{2^{m_1}}\times C_{2^{m_2}}\times \cdots\times C_{2^{m_k}}\times A_{2'}$ be an abelian group and $A_{2'}$ be the Hall $2'$-subgroup of $A$. Suppose  $H$ is a subgroup of $A$. Then $H$ is a subgroup perfect code of $A$ if and only if either $H$ is a subgroup isomorphic to 
$$C_{2^{m_1-1}}\times C_{2^{m_2-1}}\times \cdots \times C_{2^{m_k-1}}\times A_{2'}$$
or $H$ has a non-square element. 
\end{lemma}

In  the following Lemmas we discuss the regular subsets of a Cayley sum graph of $G$, which is the Cayley sum graph version of \cite[Lemma 2.7]{goh}. 
\begin{lemma}\label{1}
Let $G$ be a group and $H$ be a subgroup of $G$. Then $H$ is a $(0,\beta)$-regular set of $G$ if and only if $G$ has $\beta$ pairwise disjoint subsets $T_i$, $1\leq i\leq \beta$, such that for each $i$, $T_i\cup \{e\}$ is a  right transversal of $H$ in $G$ and $\bigcup\limits_{i=1}^\beta T_i$ is a normal  square-free subset of $G$.  In particular,  if $H$ is an  $(\alpha ,\beta)$-regular set in {\rm CayS}$(G,S)$, then $|S\cap Hx|=\beta$,  for  each $x\in G\setminus H$. 
\end{lemma}
\begin{proof}
Suppose that
$T_{1},...,T_{\beta}$
are pairwise disjoint subsets of NSq$(G)$, such that for  each $i$, where $1\leq i\leq \beta$, $T_i\cup \{e\}$ is a right transversal of $H$ in $G$ and let  $S=\bigcup\limits_{i=1}^\beta T_i$ is a normal subset of $G$. Let $x \in G \setminus H$. Then for each $i\in \{1,\dots, \beta\}$ we have  $|Hx\cap T_i|$=1,  and so  there exists $h_{i} \in H$ such that $h_{i}x \in T_{i} $.  Thus,   $x$ is adjacent to at least $\beta$ elements, $h_1,\dots , h_{\beta}\in H$ in $\rm{CayS}(G,S)$. If  there exists some $h\in H\setminus \{h_1,\dots , h_{\beta}\}$ such that $hx\in S$, then there exists $ i\in \{1,\dots, \beta\}$,   such that $hx\in T_i$. Thus, $Hx\cap T_i$ contains more than one element, which is a contradiction, as $T_i$ is a right transversal of $H$ in $G$.  As $S\cap H=\emptyset$,  we get that   $H$ is a $(0,\beta)$-regular set in CayS$(G,S)$. 

 Conversely, assume  that $H$ is a  $(0,\beta)$-regular set in CayS$(G,S)$, for some normal square-free subset $S\subseteq G$.  It means that for every $x\in G\setminus H$ there exist exactly $\beta$ distinct elements $h_1, h_2, ..., h_\beta$ in $H$ such that $h_ix\in S$, for $i=1,\dots , \beta$. Hence, $|S\cap Hx|=\beta$,  for all $x\in G\setminus H$, which means that $S$ is the union of $\beta$ pairwise disjoint subsets $T_i$, where $1\leq i\leq \beta$, such that for each $i$, $T_i\cup \{e\}$ is a right transversal of $H$ in $G$.
 \end{proof}

\begin{corollary}\label{gir}
	Let $H$ be a subgroup of $G$.  If $H$ is  an $(\alpha,\beta)$-regular set of $G$, then there is a normal square-free subset $S$ (the connection  set of the corresponding Cayley sum graph) such that $|S\cap H|=\alpha$ and $|S\cap (G\setminus H)|=\beta([G:H]-1)$.  Obviously, $\beta \leq |H|$ and  $|S|=\alpha+ \beta([G:H]-1) $. 

Moreover, there exist $\beta$  pairwise disjoint subsets $T_i$, $1\leq i\leq \beta$, such that $S\cap (G\setminus H)= \bigcup\limits_{i=1}^{\beta} T_i$, where for each $i$, $T_i\cup \{e\}$ is a right  transversal of $H$ in $G$. 
\end{corollary}

\begin{lemma}\label{normal}
Let  $H$ be a normal subgroup of $G$.  Then  $H$ is  an $(\alpha,\beta)$-regular set of $G$  if and only if  $H$ is an  $(\alpha,0)$-regular set  of  $G$ and  a  $(0,\beta)$-regular  set of $G$.
\end{lemma}
\begin{proof}
If $H$ is a normal subgroup of $G$ and  $S$  a normal square-free subset of $G$, then   $S'=S\cap H$ and $S''=S\cap (G\setminus H)$ are  normal square-free subsets of $G$.  If $H$ is  an $(\alpha,\beta)$-regular set in ${\rm CayS}(G,S)$,  then  $H$ is an  $(\alpha,0)$-regular set  in {\rm CayS}$(G,S')$ and  a  $(0,\beta)$-regular  set  in {\rm CayS}$(G,S'')$. Conversely, if $H$ is an  $(\alpha,0)$-regular set  in {\rm CayS}$(G,S')$ and  a  $(0,\beta)$-regular  set  in {\rm CayS}$(G,S'')$, for some normal square-free subsets $S'$ and $S''$, then $H$ is an  $(\alpha,\beta)$-regular set  in {\rm CayS}$(G,S'\cup S'')$. 
\end{proof}

	\begin{nota}
		
Let $G$ be a  group and $H$ be a subgroup of $G$. Then we set $$\mathcal{L}(H)=\min \{|{\rm NSq}(G) \cap Hx| :  x\in G\setminus H\}.$$

\end{nota}

\begin{theorem}\label{main}
Let $H$ be a subgroup of  an abelian group $G$. Then, $\mathcal{L}(H)= |H|$ if {\rm Sq}$(G)\subseteq  H$, and $\mathcal{L}(H)=|${\rm NS}q$(G)\cap H|$,  otherwise.
\end{theorem}

\begin{proof}
 Remind that  $\mathcal{L}(H)=\min \{|{\rm NSq}(G) \cap Hx| :  x\in G\setminus H\}.$ First, assume that  $x\in G \setminus H$ is a square element. In this case, for each non-square element  $y\in H$, $yx$ is a non-square element in $Hx$, and for each square element $y\in H$, $yx$ is a square element in $Hx$. Hence, $|$NSq$(G)\cap Hx|=|$NSq$(G)\cap H|$. 

Now, assume that $x\in G \setminus H$ is a non-square element. If $Hx$ has a square element, say $z$, then $Hx=Hz$ and, by the above discussion, we conclude that $|$NSq$(G)\cap Hx|=|$NSq$(G)\cap Hz|=|$NSq$(G)\cap H|$. If  all elements of $Hx$ are non-square, then  $|$NSq$(G)\cap Hx|=|Hx|=|H|$.  It follows that, if  Sq$(G)\subseteq  H$,   then  for every $x\in G\setminus H$, we have  $Hx\subseteq {\rm NSq}(G) $ and so $\mathcal{L}(H)=|H|$.  Otherwise,  
there exists $x\in G\setminus H$ which is square and so   $\mathcal{L}(H)=|$NSq$(G)\cap H|$. 
\end{proof}

 %\maketitle
 \maketitle

 The above  theorem is the main result of this section which has wide application for abelian groups. Using this result for given  abelian group $G$ and subgroup $H$ of $G$,  we can determine all possibilities for  $\beta$ such that  $H$ is a $(0, \beta)$-regular set of  $G$. Also it has a  significant role in giving very brief proofs for all the results related to abelian groups in  \cite{mama, ma}.

\begin{theorem}\label{first}
Let  $G$ be a  group and $H$ be a subgroup of $G$.

\begin{enumerate}

\item If $H$ is   an  $(\alpha,\beta)$-regular set of $G$, then $0\leq \alpha \leq  |{\rm NSq}(G)\cap  H|$ and  $0\leq \beta\leq \mathcal{L}(H)$. 

\item If   $G$ is abelian, then  $H$ is an $(\alpha,\beta)$-regular set of $G$ if and only if  $0\leq \alpha\leq |{\rm NSq}(G)\cap H|$ and  $0\leq \beta \leq \mathcal{L}(H)$.
\end{enumerate}
\end{theorem}

\begin{proof} {\it (1)} Let $H$ be  an $(\alpha,\beta)$-regular set in CayS$(G,S)$ for some subset $S$ of $G$ and $\mathcal{L}(H)=|$NSq$(G)\cap Hx_{0}|$, where $x_{0}\in G \setminus H$.   Then,  $S\cap Hx_0\subseteq$ NSq$(G)\cap Hx_0$. Thus, by Lemma \ref{1},  $0 \leq \beta=|S\cap Hx_0|\leq \mathcal{L}(H)$. As $\alpha =|S\cap H|$, we have $\alpha\leq |{\rm NSq(G)}\cap H|$. 

 {\it (2)}   Now, assume  that  $G$ is abelian.  By Lemma \ref{normal} we only need to prove that $H$ is  an $(\alpha, 0)$-regular set   of $G$ and a $(0,\beta)$-regular set of $G$.   Let  $0\leq \alpha\leq |{\rm NSq}(G)\cap H|$. Then by choosing $S'$ to  be an arbitrary subset of  ${\rm NSq}(G)\cap H$ of size $\alpha$ we have that $H$ is an $(\alpha,0)$-regular  set in ${\rm CayS}(G, S')$. Let  $0\leq \beta\leq \mathcal{L}(H)$. Remark that,  by the assumption each coset of $H$ in $G$, except for $H$, has at least $\mathcal{L}(H)$ non-square elements. So by choosing $\mathcal{L}(H)$ non-square elements of each coset of $H$, besides  $H$,  we are able to construct $\mathcal{L}(H)$ pairwise disjoint square-free subsets $T_i$, for $1\leq i\leq \mathcal{L}(H)$, such that for each $i$, $T_i\cup \{e\}$ is a right  transversal of  $H$ in $G$.  As $G$ is abelian, $S= \bigcup\limits_{i=1}^{\beta} T_i$ is a normal subset of $G$.  So by Lemma \ref{1}, $H$ is a $(0, \beta)$-regular set of $G$. 
\end{proof}

\begin{remark}
 It is worth mentioning that if $H$ is an $(\alpha, \beta)$-regular set of an abelian group $G$, then there are   $\binom{\mathcal {L}(H)}{\beta}^{[G:H]} \binom{|{\rm NSq}(G)\cap H|}{\alpha}$ subsets $S$ such that $H$ is an   $(\alpha, \beta)$-regular set in ${\rm CayS}(G,S)$, according to the proof of Theorem  \ref{first}. 
\end{remark}

As we mentioned before,  Theorem \ref{first} leads to the following Corollary which makes us conclude Lemmas    \ref{3.1mama} and \ref{3.1ma},  (\cite[Theorem 3.1]{mama} and \cite[Theorem 3.1]{ma}),   more simply and concisely, and so all the other results related to the abelian groups in the same papers can be obtained more easily. 

 \begin{corollary}
Let $G$ be an abelian group and $H$ be a subgroup of $G$. Then $H$ is a perfect code of $G$ if and only if {\rm Sq}$(G)\subseteq  H$ or $\rm{NSq}(G)\cap H\not=\emptyset$. 
\end{corollary} 
\begin{proof}
By  Theorem  \ref{first}, $H$ is a perfect code of $G$ if and only if $\mathcal{L}(H)\geq 1$.  On the other hand, $\mathcal{L}(H)\geq 1$ if and only if   ${\rm Sq}(G)\subseteq H$ or ${\rm NSq(G)}\cap H\not=\emptyset$, by Theorem \ref{main}. 
\end{proof}

%\begin{lemma}\label{abelian}
%Let $G$ be an abelian group and $H$ be a subgroup of $G$.  Then $H$ is an  $(\alpha,\beta)$-regular set of $G$ if and only if  $0\leq \alpha\leq |${\rm NSq}$(G)\cap H|$ and  $0\leq \beta\leq \mathcal{L}(H)$. 
%\end{lemma}
%\begin{proof}
%\end{proof}

Observe  that  if  $G$ is an abelian group and $H$ is  a $(0,\beta)$-regular set of $G$,  then for all $0\leq b\leq \beta$, $H$ is  a $(0,b)$-regular set of $G$.  Now, we prove the following theorems.  
\begin{theorem}
	Let $G$ be an abelian group and $H$ be a subgroup of $G$. Then $H$ is not a $(0,2)$-regular set of $G$ if and only if  one of the following occurs: 
	\begin{enumerate}
\item $H  \subsetneq $  {\rm Sq}$(G)$; 

\item $H$ is trivial and $G$ is an elementary abelian group; 

\item $G \cong H \times K$, where $H \cong C_{2}$ and $K$ has at least one non-trivial square element.
\end{enumerate}
\end{theorem}

\begin{proof}
	
	Suppose that $H$ is not a $(0,2)$-regular set of $G$. Theorem \ref{first} implies that  $\mathcal{L}(H)=0$ or $1$. If $\mathcal{L}(H)=0$,  then  ${\rm Sq}(G)\not\subseteq H $ and ${\rm NSq}(G)\cap H=\emptyset $, by Theorem  \ref{main}.  Thus,  $H \subsetneq$ Sq$(G)$. Now, assume that $\mathcal{L}(H)=1$. By Theorem \ref{main}, we get the following cases:
	
	Case 1: $\mathcal{L}(H)=|H|=1$. 
	
	In this case, Sq$(G)\subseteq  H=\{e\}$ and so every non-trivial element of $G$ is non-square, implying that the order of each non-trivial element is $2$. So $G$ is an elementary abelian $2$-group.
	
	Case 2: $\mathcal{L}(H)=|$NSq$(G)\cap H|=1$. 
	
	Let NSq$(G)\cap H =\{x\}$. As  $x^3\in {\rm NSq}(G)\cap H$,  we conclude that $x=x^3$ and so  the order of $x$ is $2$. If $y\in H  \setminus \{x\}$, then $y$ is a square and so $xy \in$ NSq$(G)\cap H$, which implies that $y=e$. Therefore, $H=\{e, x\}\cong C_{2}$. We may assume  that   $G = G_0\times G_1\times \cdots \times G_k\times A $, where  $k\geq 0$ and $G_i \cong C_{2^{\alpha_{i}}}$, for some integers $\alpha_{0}, ... ,\alpha_{k}$, and $A$ is the Hall $2'$-subgroup of $G$. Let $x=(l_0, l_1,\dots, l_k, a)$,  where $a\in A$ and  $l_i\in G_i $,  for $i=0,\dots, k$.  As $x^2=e$, we conclude that  $a=e $ and $o(l_j)\leq 2$, for each $j=0,\dots, k$.   As $x$  is non-square,  there exists $0\leq i_0\leq k $ such that $l_{i_0}$ is non-square, and so $l_{i_0}$ is a generator of $G_{i_0}$ of order $2$. So without loss of generality, we may assume that  $i_0=0$ and $c=l_{i_0}=l_0$. Then  $G_0=\{e,c\} $ and  $x=(c, l_1,l_2, ... , l_k,e)$.   Setting  $K=\{e\}\times G_1\times \cdots \times G_k\times A$,  we have $  H\cap K$ is trivial  and consequently,  $G = H\times K$.    	
	\end{proof}

%\begin{remark}

%\end{remark}

\begin{theorem}
	Let $G$ be an abelian group and $H$ be a subgroup of $G$. Then $H$ is not a $(0,3)$-regular set of $G$ if and only if one of the following occurs:  
		\begin{enumerate}
	\item $H \subsetneq$ {\rm Sq}$(G)$;
	
\item  $G$ is an elementary abelian $2$-group and $|H| \leq 2$;

 \item  $G \cong H\times K$, where $H\cong C_2 $ and $K$ has at least one non-trivial square element;
	
\item   $G \cong \bold{E_2}(G)\times T$, where $T\cong C_{4}$,  and $H$ is the subgroup of order $2$ of $T$;

\item  $G\cong \bold{E_2}(G) \times K$ and   $H=\langle x,y\rangle$ where  $x=(m, k_1)$ and $y=(m, k_2)$, such that   $e\not =m \in \bold{E_2}(G)$  and $k_1, k_2\in K$   have orders at most $2$. 

  \item  $H$ is generated by a non-square element of order $4$;
       \end{enumerate} 
	
\end{theorem}

\begin{proof}
	Suppose that $H$ is not a $(0,3)$-regular set of $G$.  Theorem  \ref{first} implies that $\mathcal{L}(H) \in \{0,1,2\}$. If $\mathcal{L}(H)=0$ or $1$, then the result is the same as the previous theorem and  $H$ satisfies one of the Cases (1)-(3).   So suppose that $\mathcal{L}(H)=2$.

First, assume that Sq$(G) \subseteq  H$, hence $\mathcal{L}(H)= |H|=2$. In this case, if $|{\rm Sq}(G)|=1$, then   $G \cong C_2^{\alpha}$, for some $\alpha$ and $H$ is any subgroup of $G$ of order $2$, which is Case (2). If $|$Sq$(G)|=2$, then, as  all  elements in the Hall $2'$-subgroup of $G$ are square,  we deduce that $G =  \bold{E_2}(G) \times T$, where $T\cong C_4$ and  $H $ is the unique  subgroup of order $2$ of $T$, which is Case (4).

	Now, suppose that $\mathcal{L}(H)=|$NSq$(G)\cap H|=2$. Let NSq$(G)\cap H=\{x,y\}$. Then $x^3\in H$ is a non-square element. Therefore, either $x^3=x$ or $x^3=y$. 

$\bullet$ First, assume  $x^3=x$,   and so $o(x)=2$. Let $z\in H\setminus \{x,y\}$. Then, $xz\in {\rm NSq}(G)\cap H$ and hence $xz=x$ or $xz=y$. Thus, either $z=e$ or   $z=x^{-1}y$, which means  that $|H|=4$.  As $o(x)=2 $ and $x$ is non-square, then $H\cong C_2\times C_2$.  
 Assume  $G = G_0\times G_1\times \cdots \times G_n\times A $, where  $n\geq 0$ and $G_i \cong C_{2^{\alpha_{i}}}$, for some integers $\alpha_{0}, ... ,\alpha_{n}$, and $A$  is  the Hall $2'$-subgroup of $G$.  Let $x=(l_0, l_1,\dots, l_n, a)$,  where $a\in A$ and  $l_i\in G_i $,  for $i=0,\dots, n$.  As $x^2=e$, we conclude that  $a=e$ and $o(l_j)\leq 2$, for each $j$. Since  $x$  is non-square,  there exists $0\leq i_0\leq n $ such that $l_{i_0}$ is a non-square element of $G_{i_0}$, and so  $l_{i_0}$ is the generator of $G_{i_0}\cong C_2$. 
% As $x$ is of order $2$, by the same discussion that we have in the proof of the  previous Lemma (where $x$ is a non-square element of order $2$),
 As  $G=\bold{E_2}(G)\times K$, for some subgroup $K$ of $ G$, we   conclude that there exists $0\not =m\in \bold{E_2}(G)$  and $k_1\in K$  such that  $x=(m, k_1)$, where $o(k_1)\leq 2$.   Now, consider $z=x^{-1}y\in H$. As $z$ is square and it has order $2$ then $z=(e, k)\in \bold{E_2}(G)\times K$, for some $k\in K$ of order $2$. So $y=xz=(m, k_1k)$ has the form described in Case (5). 

$\bullet$ Now, assume $x^3=y$.  If $z \in H\setminus \{e,x,y\}$, then, as $xz$ is non-square and $z$ is non-trivial, we deduce that  $xz=x^3=y$ and so $z=x^2$. Therefore $H$ is the  cyclic group of order 4, generated by $x$, which is Case (6).
\end{proof}

\begin{remark}
Note that in Case (6), $H$ is not necessarily a direct factor of $G$. For instance,  $G=C_2\times C_8$ and $H=\langle (1,2)\rangle$ is satisfying Case (6), however $H$ is not a direct  factor of $G$. 
\end{remark}

\section{dihedral groups}

 Throughout this section we assume  $n\geq 3$,   $G=D_{2n}=\left\langle a,b: a^{n} = b^{2} = e, b^{-1}ab = a^{-1}\right\rangle$,  the dihedral group of order $2n$. We note that $o(a^{i}b)=2$, for each $1 \le i \le n$ and $a^sba^kb^l=a^{s-k}b^{l+1}$, for all integers  $s,k$ and $l$. The subgroups of $G$   (see for example  \cite{estef}) are  as follows:

\begin{itemize}
 \item   the cyclic subgroup $H= \left\langle a^{t}\right\rangle $,   where $t$ divides $n$;  and  so $[G: H ]=2t$;  

\item  the subgroup   $H=\left\langle a^{t},a^{s}b \right\rangle $, where $t$ divides $n$ and $0 \leq s \leq t - 1$; which is   isomorphic to $C_2$ (when $t=n$),  $C_2\times C_2$ (when  $n$ is even and $t=n/2$) or the  dihedral group $D_{2n/t}$ (when $t$ is a proper divisor of $n$ and $t\not = n/2$).  In this case   $[G:H]=t$. 
\end{itemize}

 We use the above notations without further references.

\begin{lemma} \cite{conjugacy} \label{3.1} 
	 Suppose that $n = 2m$, for some positive integer $m > 3$. Then $G=D_{2n}$ has $m + 3$ conjugacy classes as follows:	 
	 $ \{e\}$, $\{a^{m}\}$, $b^G=\{a^{2 j}b : 0 \leq j \leq m-1\}, $  $(ab)^G=\{ a^{2j+1}b : 0 \leq j \leq  m-1\}, $ $\{a^i , a^{-i} \}$,
 where $1 \leq i \leq m-1$.
\end{lemma}

\begin{lemma} \cite{conjugacy} \label{odd n} 
	 Suppose that $n$ is an odd integer. Then $G=D_{2n}$ has $(n+3)/2$ conjugacy classes as follows:	 
	 $ \{e\}, $ $ b^G=\{a^{j}b : 0 \leq j \leq n-1\},$ $ \{a^{i} , a^{-i} \},$
 where $1\leq i \leq (n-1)/2$.
\end{lemma}

Note that, for each  $1 \le i \le n$, $a^{i}b$ is a non-square element in $D_{2n}$. If $n$ is odd, then $a^i$ is square,  for $i=1,\dots,  n$. Also, 
if $n$ is even, then $a^{2i}$ is  square and $a^{2i+1}$ is non-square, for $i=0,\dots, (n-2)/2$.

\begin{example} For  comprehending the topic, we describe some examples of an  $(\alpha, \beta)$-regular set  of dihedral group  $G=D_8$: 

$\bullet$
 Let $H=\left\langle a^2b \right\rangle$. Then $[G:H]=4$ and the set of cosets of $H$ in $G$ is  $\{H, Ha, Ha^2, Ha^3\}$. In this case,  $ {\rm NSq}(G) \cap H= \{ a^2b\}$, $ {\rm NSq}(G) \cap Ha=Ha = \{a, ab\}$,     ${\rm NSq}(G) \cap Ha^2=\{b\}$ and ${\rm NSq}(G) \cap Ha^3=Ha^3=\{a^3, a^3b\}$. So, $\mathcal{L}(H)=1$ and  by Theorem  \ref{first}$, \alpha , \beta \leq 1$. Then,   $S= b^G \cup (ab)^G$ is a normal square-free subset of $G$ and $|S \cap Hx|=1$,  for all $x\in G$.  By  Lemma \ref{1} and  Corollary \ref{gir},  $H$  is a $(1,1)$-regular set in CayS$(G, S)$.  For more details see  Figure 1(a),

% In this grap, h the vertices in $H$ are  shown by red nodes and  the  edges between two vertices in $H$ are shown by green  and the edges between a vertex  of $H$   and a vertex out of $H$ are shawn by blue.  

$\bullet$
Let $H=\left\langle a^2, ab \right\rangle$. Then  $[G:H]=2$. Note that, $(ab)^G$ is the only normal square-free subset of $G$ contained in $H$ and the square-free conjugacy classes of $G$ contained in $G \setminus H$ are $\{a, a^3\}$ and $b^G$. Set $S^{\prime \prime}= \emptyset$, $\{a, a^3\} $ or $ b^G$ or $\{a, a^3\}\cup b^G$.
Setting  $ S=(ab)^G \cup S^{\prime\prime}$, we have  $H$ is a $(2, |S^{\prime\prime}|)$-regular set in CayS$(G, S)$. For instance $H$ is a $(2,2)$-regular set in ${\rm CayS}(G, (ab)^G\cup b^G)$, which can be seen in Figure (b).

In  these  graphs, for both examples,  the vertices in $H$ are  blank,  the edges between two vertices in $H$ are dotted line  and the edges between a vertex  of $H$   and a vertex  of $G\setminus H$ are  bold.   

\end{example}

\begin{figure}[h]
\begin {center}
\begin{tikzpicture}[circ/.style={circle, draw, fill}]

\vertex[] (v1) at (0.5,0) [] [label=above:$e$] {};
\vertex[] (v2) at (2,0) [][label=above:$a^2 b$] {};

\vertex[fill] (v3) at (-1,-1) [label=left:$a$] {};
\vertex[fill] (v4) at (-1,-2) [label=left:$a^2$] {};
\vertex[fill] (v5) at (-1,-3) [label=left:$a^3$] {};

\vertex[fill] (v6) at (3.5,-1) [label=right:$b$] {};
\vertex[fill] (v7) at (3.5,-2) [label=right:$ab$] {};
\vertex[fill] (v8) at (3.5,-3) [label=right:$a^3b$] {};

\path (v1) edge [dashed]  (v2);

\path (v3) edge  (v8);
\path (v3) edge [ultra thick](v2);
\path (v3) edge  (v7);
\path (v3) edge  (v6);

\path (v4) edge  (v8);
\path (v4) edge[ultra thick] (v2);
\path (v4) edge  (v7);
\path (v4) edge  (v6);

\path (v5) edge  (v8);
\path (v5) edge [ultra thick] (v2);
\path (v5) edge  (v7);
\path (v5) edge  (v6);

\path (v1) edge[ultra thick](v8);
%\path (v1) edge  (v2);
\path (v1) edge [ultra thick]  (v7);
\path (v1) edge [ultra thick] (v6);

\end{tikzpicture}
\hspace{2cm}
\begin{tikzpicture}[circ/.style={circle, draw, fill}]

\vertex[] (v1) at (0.5,0) [] [label=above:$e$] {};
\vertex[fill] (v2) at (2,0)[label=above:$a^2 b$] {};

\vertex[fill] (v3) at (-1,-1) [label=left:$a$] {};
\vertex[] (v4) at (-1,-2) [][label=left:$a^2$] {};
\vertex[fill] (v5) at (-1,-3) [label=left:$a^3$] {};

\vertex[fill] (v6) at (3.5,-1) [label=right:$b$] {};
\vertex[] (v7) at (3.5,-2)[] [label=right:$ab$] {};
\vertex[] (v8) at (3.5,-3)[] [label=right:$a^3b$] {};

\path (v1) edge [ultra thick]  (v2);

\path (v3) edge [ultra thick](v8);
\path (v3) edge (v2);
\path (v3) edge [ultra thick] (v7);
\path (v3) edge  (v6);

\path (v4) edge  [dashed](v8);
\path (v4) edge [ultra thick](v2);
\path (v4) edge [dashed] (v7);
\path (v4) edge  [ultra thick](v6);

\path (v5) edge[ultra thick]  (v8);
\path (v5) edge   (v2);
\path (v5) edge [ultra thick] (v7);
\path (v5) edge  (v6);

\path (v1) edge [dashed] (v8);
%\path (v1) edge  (v2);
\path (v1) edge [dashed]  (v7);
\path (v1) edge[ultra thick] (v6);

\end{tikzpicture}
\end{center}
\caption{(a) (1,1)-regular set \hspace{4cm}  (b) (2,2)-regular set }
\end{figure}

\begin{theorem}
Let  $H=\langle a^sb \rangle$, where $0\leq s\leq n-1$,  be a subgroup of $G=D_{2n}$. Then $H$ is an $(\alpha, \beta)$-regular set of $G$, if and only if $(\alpha, \beta)=(1, 1)$.
\end{theorem}

\begin{proof}
 Let  $H$ be an $(\alpha,\beta)$-regular set in ${\rm CayS}(G, S)$,  for some subset $S$ of $G$.   Note that $Ha^k=\{a^{k}, a^{s-k}b\}$, for each integer $k$.  Therefore, ${\rm NSq}(G)\cap Ha^{2i}=\{ a^{s-2i} b\}$,  for each integer $i$.  Also, if $n$ is even, then   ${\rm NSq}(G)\cap Ha^{2i+1}=\{ a^{2i+1}, a^{s-2i-1}b \}$ and   if $n$ is odd, then  ${\rm NSq}(G)\cap Ha^{2i+1}=\{ a^{s-(2i+1)}b\}$, for each integer $i$.   Therefore, by Theorem \ref{first}, $\alpha\leq  1$ and   $\beta\leq \mathcal{L}(H)=1$. Now, we claim that if $\beta=1$,  then $\alpha=1$.  Let $\beta=1$. 
By Lemma \ref{1},   $a^{s-2i}b\in  S$,  for each $i$. Therefore, as $S$ is a normal subset of $G$,  by Lemmas \ref{3.1} and \ref{odd n}, we get that  $a^sb\in (a^{s-2i}b)^G\subseteq S$ (as the parity of $s$ and $s-2i$ are the same), and so $\alpha=|S \cap H |=1$.

If $\beta=0$, then,  by Lemma \ref{1},   for each $1\leq i\leq n-1$,  we have  $|S\cap Ha^i|=0$. So  $a^{s-2}b\not \in S$.  As $a^sb\in (a^{s-2}b)^G$ and $S\cap (a^{s-2}b)^G=\emptyset$, it follows that $\alpha=0$ and $S=\emptyset$, which is excluded. Thus, the only possibility is  $(\alpha, \beta)= (1, 1)$.

  In addition, by taking $S=(ab)^G\cup b^G$, we get that $H$ is a  $(1,1)$-regular set of $G$. 
\end{proof}

\begin{theorem}
	Let $n$ be an odd integer,   $t>1$ a  proper divisor of $n$, and  $0 \leq s\leq t-1$.  Let  $H=\left\langle a^{t}, a^{s}b \right\rangle$ be a subgroup of $G=D_{2n}$. Then $H$ is an $(\alpha, \beta)$-regular set  of  $G$ if and only if $(\alpha, \beta)=(|H|/2, |H|/2)$.
\end{theorem}
\begin{proof}
 By the assumption,    $H$ is isomorphic to $D_{2n/t}$. Let $H$ be an   $(\alpha, \beta)$-regular set   in  {\rm CayS}$(G,S)$,  for some  subset  $\emptyset\not =S\subseteq G$.  Since  $S\subseteq {\rm NSq}(G)=b^G$ and  $S$ is normal,  we get that  $S=b^G$.   Note that  $H=\{a^{rt}, a^{rt+s}b \ |\   \text{where $1\leq r\leq n/t$} \}$  and so  for every integer $k\in \{0,\dots, n-1\}$ and $l\in \{0,1\}$,  
$Ha^{k}b^l=\{a^{rt+k}b^l, a^{s+rt-k}b^{l+1} | \ \text{where $1\leq r\leq n/t$}\}.$ 
  We can see that   $|b^G \cap Ha^{k}b^l|=|H|/2$. Therefore,   $\alpha=|S\cap H|=| b^G\cap H|=|H|/2$,  and $\beta=|S \cap Hx|=| b^G \cap Hx|=|H|/2$,  for each $x\in G\setminus H$. 
%As $H$ is a dihedral group and  NSq$(G) \subseteq b^G$,  we have $|$NSq$(G)\cap Ha^{l}b^k|=|H|/2$, for every integer $l\in \{0,\dots, n-1\}$ and $k\in \{0,1\}$.  If $$ $|Ha^{l}\cap S|=|H|/2$, for each $0 \leq l \leq n-1$. Therefore, either $\beta=0$ or $|H|/2$. If $\alpha \ne 0$ or $\beta \ne 0$, then $b^G \subseteq S$ and so $\alpha=\beta=|H|/2$. So $(\alpha,\beta)\in \{(0,0), (|H|/2, |H|/2)\}$.
\end{proof}
\begin{theorem}\label{3.5}
	Let $n=2m \geq 4$,  for some integer $m$,  $t>1$ is a proper divisor of $n$ and  $0 \leq s\leq t-1$.  Suppose $H=\left\langle a^{t}, a^{s}b \right\rangle$ is  a subgroup of $G=D_{2n}$. Then $H$ is an $(\alpha, \beta)$-regular set of $G$,  $(\alpha,\beta)\not=(0,0)$,  if and only if one of the following occurs:

1)  $\alpha\in \{0, |H|/2\}$ and $0\leq \beta\leq |H|$, when $t=2$ and $m$ is odd;
 
2)  $\alpha \in \{0,|H|/2\}$ and $0 \leq \beta=2\gamma \leq |H|$, for some integer $\gamma$, when $t=2$ and $m$ is even;
 
3) $(\alpha, \beta)=(|H|/2, |H|/2)$,  when  $t>2$ is even;
	
	4) 	   $(\alpha, \beta)\in  \{(\eta, \zeta), (\eta+m/t, \zeta+m/t), (\eta + 2m/t, \zeta +2m/t) \ : \  0\leq \eta, \zeta \leq m/t \}$,  when  $m$ is odd and $t$ is a divisor of $m$;
	   
	   5) 
	   $(\alpha, \beta)\in  \{(2\eta, \zeta), (2\eta+m/t, \zeta+m/t), (2\eta + 2m/t, \zeta +2m/t) \ : \    0\leq \eta \leq m/2t ,  \ 0\leq  \zeta \leq m/t \}$, when  $m$ is even and $t$ is odd.
\end{theorem}
\begin{proof} 

  For getting the result, first we assume that $H$ is an   $(\alpha, \beta)$-regular set   in  {\rm CayS}$(G,S)$,  for some  subset  $\emptyset\not =S\subseteq G$ and  we find some restrictions on   $(\alpha, \beta)$. Then for each $(\alpha, \beta)$ satisfying those restrictions we give a normal square-free subset $S$ such that   $H$ is an   $(\alpha, \beta)$-regular set   in  {\rm CayS}$(G,S)$.

 By the assumption,    $H=\{a^{rt}, a^{rt+s}b \ |\   \text{where $1\leq r\leq n/t$} \}$ is isomorphic to $D_{2n/t}$ or $C_2\times C_2$. 
We consider the following cases:

{\bf Case 1.} Let $t$ be even. 
%\begin{enumerate}

{\bf Subcase 1.1.} Let $t=2$. Then $[G : H]=2$. Thus,  $H$  is normal in $G$ and so by the proof of  Lemma \ref{normal}, we get that $H$ is an $(\alpha,0)$-regular set in ${\rm CayS}(G, S\cap H)$ and  $H$ is a $(0, \beta)$-regular set in ${\rm CayS}(G, S\cap (G\setminus H))$.    

First, we find the possibilities for $\alpha$ and $\beta$. 
Note that, $H=\left\langle a^{t}, a^{s}b \right\rangle$ contains $(a^{s}b)^G$, which is the only normal square-free subset of $G$ contained in $H$, and $|(a^{s}b)^G|=|H|/2$. Then $S'=S\cap H$ either is the empty-set or is equal to $(a^{s}b)^G$ and so $\alpha=|S\cap H|\in \{0, |H|/2\}$. 
Also,   as $|G:H|=2$, we have   $0\leq \beta=|S\cap Hx|=|S\cap (G\setminus H)|\leq |G\setminus H|=|H|$, for every $x\in G\setminus H$.

 In the sequel, we show that    if  $m$ is odd, then for each $\beta\in \{0,\dots, |H| \}$, there exists  a  normal square-free subset of  size $\beta$   in $G\setminus H$,   and if $m$ is even, then for each even integer  $\beta\in \{0,\dots,|H|\}$  there exists  a  normal square-free subset of  size $\beta$   in $G\setminus H$. In both cases, we denote them  by $S_{\beta}$.  

$\bullet$
Assume that $m$ is odd. Then  the  square-free conjugacy classes of $G$ contained in  $G\setminus H$ are as follows:  
 $\{a^m\}$, $(a^{s+1}b)^G$ (of size $|H|/2$) and $\mathcal{S}_{i}=\{a^{2i+1}, a^{-(2i+1)}\}$, where $0 \le i < (|H|/2-1)/2$. 
 If $0\leq \beta \le|H|/2$ is odd, then  let  $S_{\beta}=\{a^m\} \cup \Big( \bigcup \limits_{i=1}^{(\beta-1)/2} \mathcal{S}_i \Big)$.   If $0\leq \beta< |H|/2$ is even, then  let $ S_{\beta}= \bigcup \limits_{i=1}^{\beta/2} \mathcal{S}_i$. Now, let $\beta > |H|/2$ and $\gamma = \beta -|H|/2$.  If $\gamma$ is even, then  let $S_{\beta}=  (a^{s+1}b)^G \cup  \Big(\bigcup \limits_{i=1}^{\gamma/2} \mathcal{S}_i\Big)$. If $\gamma$ is odd, then  let $S_{\beta}= \{a^m\} \cup (a^{s+1}b)^G \cup  \Big(\bigcup \limits_{i=1}^{(\gamma-1)/2} \mathcal{S}_i\Big)$. 

 Let    $0\leq \beta\leq |H|$. So, in this case $H$ is  a $(0, \beta)$-regular set  in ${\rm  CayS}(G, S_{\beta})$ and  $H$ is  a  $(|H|/2, \beta)$-regular  set  in ${\rm  CayS}(G, (ab)^G\cup S_{\beta})$, (except for  $(\alpha, \beta)=(0,0)$), which is Part (1) of the statement of the theorem.

 $\bullet$
Suppose that $m$ is even. Then $a^m\in H$ and  the  square-free conjugacy classes of $G$ contained in  $G\setminus H$ are as follows:  
 $(a^{s+1}b)^G$ (of size $|H|/2$) and $\mathcal{S}_{i}=\{a^{2i+1}, a^{-(2i+1)}\}$, where $0 \le i < (|H|/2-1)/2$. As  $S$ is a union of some  square-free conjugacy classes of $G$ and all conjugacy classes of $G \setminus H$ have even size,  $\beta=|S \cap (G\setminus H)|$ is even. Hence, by similar discussion as above for every even number $0\leq \beta \leq |H|$ we can construct a normal square-free subset $S_{\beta}$ of size $\beta$.  

  Let    $0\leq \beta\leq |H|$ be even. So, in this case $H$ is  a $(0, \beta)$-regular set  in ${\rm  CayS}(G, S_{\beta})$ and  $H$ is  a  $(|H|/2, \beta)$-regular  set  in ${\rm  CayS}(G, (ab)^G\cup S_{\beta})$, (except for  $(\alpha, \beta)=(0,0)$), which is Part (2) of the statement of the theorem.

%So, in this case $H$ is $(\alpha, \beta)$-regular for all  even  $0\leq \beta\leq |H|$ and  $\alpha\in \{0,|H|/2\}$ (except for  $(\alpha, \beta)=(0,0)$, as $S\not=\emptyset$). 

{\bf Subcase 1.2} Let $t\geq 4$.  We note that $Ha^{k}=\{a^{rt+k}, a^{s+rt-k}b \ :\ \text{where $1\leq r\leq n/t$}\}$. In this case, $|$NSq$(G)\cap Ha^{2l}|=|H|/2$, $|$NSq$(G)\cap Ha^{2l+1}|=|H|$, where $0 \leq l \leq m-1$. Since $Ha^rb=Ha^{s-r}$, for each $0\leq r\leq n-1$, we get that $\mathcal{L}(H)=|H|/2$. Thus, according to Theorem \ref{first}, $0\leq \alpha\leq |H|/2$  and   $0\leq \beta\leq |H|/2$. In the sequel,  we prove that the only possibility is $\alpha=\beta=|H|/2$.

Note that  ${\rm NSq}(G)\cap Ha^{2l}\subseteq (a^sb)^G$ and  if  $\beta=0$, then $S \cap Ha^{2} = \emptyset $.  By the normality of $S$,   we have  $S \cap (a^{s}b)^G = \emptyset$. So $S \cap H= \emptyset$, implies that $\alpha=0$, which is excluded, as $S\not =\emptyset$ . 

Let $\beta=|S\cap Ha^{2}|\geq 1$. As  NSq$(G)\cap Ha^{2} \subseteq (a^{s}b)^G$,   by the normality of $S$, we conclude that $  (a^{s}b)^G \subseteq S$. On the other hand,  ${\rm NSq(G)}\cap H\subseteq (a^sb)^G\subseteq S$. Thus, $\alpha=|S\cap H|= |{\rm NSq}(G) \cap H|=|H|/2$ and $\beta=|{\rm NSq}(G)\cap Ha^{2}|=|H|/2$. So the only possibility is $\alpha=\beta=|H|/2$.   On the other hand,  taking   $S=(a^{s}b)^G \cup (a^{s+1}b)^G=b^G\cup (ab)^G$,  leads to the conclusion that $|S\cap Ha^{k}|=|H|/2$, for each $0\leq k\leq n-1$  and so $H$ is a $(|H|/2, |H|/2)$-regular set in ${\rm CayS}(G,S)$, which is Part (3) of the statement of the theorem.

%Note that $Ha^{2l+1} \subseteq (a^{s+1}b)^G \cup \{a^k | k$ is odd and $0 \leq k \leq n-1\}$, $|(a^{s+1}b)^G \cap Ha^{2l+1}|= | \{a^k | k$ is odd and $0 \leq k \leq n-1\}\cap Ha^{2l+1} |= |H|/2$. Therefore, as $|S \cap Ha^{2l+1}|= |H|/2$, either $(a^{s+1}b)^G\subseteq S$ or $\{a^k | k$ is odd and $0 \leq k \leq n-1\} \subseteq S$. Thus, $S=(a^{s+1}b)^G \cup (a^{s}b)^G$ or $S=\{a^k | k$ is odd and $0 \leq k \leq n-1\} \cup (a^s b)^G$.  

{\bf  Case 2.} Let $t$ be odd. Thus, $${\rm NSq}(G)\cap H=\{a^{kt}, a^{-kt}: \text{ where $1\leq k\leq  m/t$ is odd} \} \cup ( b^G\cap H) \cup ( (ab)^{G}\cap H). $$ 

  In Subcases 2.1 and 2.2, we will determine $(\alpha, \beta)$,  when $m$ is even and odd, respectively. Before that,  we prove the following claims, in order to make the subcases easier to follow.    Let $A_{i}=\{a^{it}b, a^{it+1}b, \dots, a^{(i+1)t-1}b\}$, where $0 \leq i \leq (2m/t)-1$.  

 {\bf Claim 1:} {\it  $ |b^{G}\cap H|= | (ab)^{G}\cap H|=m/t$. }

$\bullet $ Let $s$ be odd. Then $rt+s$ is even if and only if $r$ is odd; and $rt+s$ is odd if and only if $r$ is even. Thus by Lemma \ref{3.1}, $ b^{G}\cap H =\{a^{(2i+1)t+s}b \ : \  0\leq i\leq (m/t)-1\} \subseteq \bigcup\limits_{i=0}^{(m/t)-1} A_{2i+1}$ and $ (ab)^G\cap H=\{a^{2it+s}b \ : \  0\leq i\leq (m/t)-1\} \subseteq \bigcup\limits_{i=0}^{(m/t)-1} A_{2i}$. We note that, $ b^G  \cap A_{2i+1}\cap H=\{a^{(2i+1)t+s}b\}$ and $ (ab)^G \cap A_{2i}\cap H=\{a^{2it+s}b\}$, for each $i$. 
	 
 $\bullet$ Now, let  $s$  be even.  Then,   similarly,  $ b^{G}\cap H \subseteq \bigcup\limits_{i=0}^{(m/t)-1} A_{2i}$, $(ab)^{G} \cap H\subseteq \bigcup\limits_{i=0}^{(m/t)-1} A_{2i+1}$, $b^G \cap A_{2i}\cap H=\{a^{2it+s}b\}$ and $ (ab)^G \cap A_{2i+1}\cap H=\{a^{(2i+1)t+s}b\}$,  for $0 \leq i \leq (m/t)-1$. 

Thus, in both cases,  $ |b^G\cap H|=| (ab)^G\cap H|=m/t$, as claimed.  $\Box$
%\end{proof}
%Therefore, $| (G\setminus H)\cap b^G|=|(G\setminus H)\cap (ab)^G|=m-m/t$.
 
 {\bf Claim 2:}  {\it  $|b^G \cap Hx |=|(ab)^G \cap Hx|=m/t$, for every $x \in G \setminus H$.}

 To prove our claim we just need to show that $b^G \cap (G\setminus H) =\bigcup\limits_{j=1}^{m/t} T_j$ and $(ab)^G \cap (G\setminus H) =\bigcup\limits_{j=1}^{m/t} {\mathcal{T}}_j$, for some square-free pairwise disjoint subsets $T_j$ and $\mathcal{T}_j$ such that for each $1\leq j\leq m/t$, $T_j\cup \{e\}$ and $\mathcal{T}_j\cup \{e\}$ are right transversals of $H$ in $G$. 
In the sequal of this paragraph,  we take $i$ to be an integer in  $\{0,\dots, (m/t)-1\}$.  Remind that   $H=\{a^{rt}, a^{rt+s}b \ |\   \text{where $1\leq r\leq n/t$} \}$.  As  $|H\cap A_i|=1 $, for each $0\leq i\leq (2m/t)-1$,  we have $|(A_{2i}\cup A_{2i+1})\setminus H|=2(t-1)$ and so $|(A_{2i}\cup A_{2i+1}) \cap b^{G} \setminus H|=|(A_{2i}\cup A_{2i+1}) \cap (ab)^{G} \setminus H|=t-1$. Set $T_{i+1}= (A_{2i}\cup A_{2i+1})\cap b^G\setminus H$  and $\mathcal{T}_{i+1}= (A_{2i}\cup A_{2i+1})\cap (ab)^G\setminus H$. Then it is enough to prove that if $x_1, x_2\in T_{i+1}$ or $x_1, x_2\in \mathcal{T}_{i+1}$ are distinct elements, then $Hx_1\not =Hx_2$. On the contrary, assume that $Hx_1=Hx_2$. If $x_1, x_2\in A_{k}$, where $k\in \{2i, 2i+1\}$, then $x_1=a^{kt+r_1}b$ and $x_2=a^{kt+r_2}b$, for some $0\leq r_1< r_2\leq t-1$. Thus $a^{r_2-r_1}\in H$, which is not possible. If $x_1\in A_{2i}$ and $x_2\in A_{2i+1}$, then $x_1=a^{2it+r_1}b$ and $x_2=a^{(2i+1)t+r_2}b$, for some $0\leq r_1, r_2\leq t-1$. Hence, $a^{t+r_2-r_1}\in H$, which means that $a^{r_2-r_1}\in H$. It follows that $r_1=r_2$.  On the other hand, the parities  of $2it+r_1$ and $(2i+1)t+r_1$ are  the same,  as $a^{2it+r_1}b$ and $a^{(2i+1)t+r_1}b$ both are in the same conjugacy class $b^G$ or $(ab)^G$. This implies that   the parities of $r_1$ and $t+r_1$ are  the same.  Thus, $t$ is even, a contradiction.  So the proof is complete.  $\Box$

{\it {\bf Claim 3:}    For each $1 \leq \beta \leq m/t$, let  $\Omega_{\beta}=\bigcup\limits_{i=0}^{\beta-1}\mathfrak{S}_{i}$, where   $\mathfrak{S}_i= \{a^{j},a^{-j}: it+1 \leq j \leq (i+1)t-1,  \text{ where} \ j$  is odd$\}$.  
%Then $\mathfrak{S}_i\cup\{e\}$ is a right transversal of $H$ in $G$, for $0\leq i\leq m/t-1$.  
Then,     $H$ is a   $(0,\beta)$-regular  set in {\rm CayS}$(G, \Omega_{\beta})$.}

 First, we prove that $\mathfrak{S}_i\cup\{e\}$ is a right transversal of $H$ in $G$, for $0\leq i\leq m/t-1$.  We  show that  for any pair $ it+1 \leq j_1 \leq j_2 \leq (i+1)t-1$,  we have  $Ha^{j_1}\not  =Ha^{- j_2}$, and if $Ha^{j_1}=Ha^{j_2}$, then  $j_1=j_2$. If $Ha^{j_1} =Ha^{- j_2}$, then  $a^{j_1+j_2}\in H$ and so  ${j_1+j_2}=kt $,  for some integer $k$.   On the other hand, $j_1+j_2$ is even and $t$ is odd, implying that $t$ divides $(j_1+j_2)/2$, a contradiction,  as  $j_1/2 ,   j_2/2\in [(it+1)/2, ((i+1)t-1)/2] $.  If  $Ha^{j_1}=Ha^{j_2}$, then $t$ divides $j_1-j_2$, implying that $j_1=j_2$. 
 Also $Hx\not =H$, for $ x\in \mathfrak{S}_i$.  Therefore, as $|\mathfrak{S}_i|=t-1$ we conclude that   $\mathfrak{S}_i\cup \{e\}$ is a right transversal of $H$ in $G$.  Thus, as $\mathfrak{S}_i$  is a normal square-free subset of $G$, by  Lemma \ref{1}, 
 $H$ is a $(0,\beta)$-regular set in CayS$(G,\Omega_{\beta})$, for all $0\leq \beta\leq m/t$, as wanted. $\Box$

In the following subcases, we use the notation $\Omega_{\beta}$ which is defined in  Claim 3, without further references.  Also, by $\Omega_0$ we mean the empty set. 

{\bf Subcase 2.1.} Let $m$ be odd. Now we consider the following cases for $S$.

$\bullet$
Suppose that $S \cap (b^{G} \cup (ab)^{G})= \emptyset $.  Therefore,  $S\subseteq T=\{a^m\}\cup \left(\bigcup\limits_{i=0}^{(m-3)/2} \{a^{(2i+1)}, a^{-(2i+1)}\} \right)=\{a^m\}\cup \Omega_{m/t}\cup \left(\bigcup\limits _{j=0}^{((m/t)-3)/2} \{a^{t(2j+1)}, a^{-t(2j+1)}\} \right)$.   Remind that   $\{a^m\}$ and $\{a^{tk}, a^{-tk}\}$, for every odd integer  $k< m/t$, are square-free conjugacy classes of $G$ contained in $H$.   Then $0\leq \alpha=|S\cap H|\leq  |T\cap H|=\left|\{a^m\}\cup \left(\bigcup\limits _{j=0}^{((m/t)-3)/2} \{a^{t(2j+1)}, a^{-t(2j+1)}\} \right)\right|=m/t$ and $0\leq \beta=|S\cap Hx|\leq |\Omega_{m/t}\cap Hx|=m/t $, for each $x\in G\setminus H$,  by Claim 3.

Now, we show that, in this case,   for each  $\alpha, \beta\in \{0,\dots,m/t\}$, there exists  a normal square-free subset  $S$ such  that  $H$ is an $(\alpha, \beta)$-regular set in ${\rm CayS}(G,S)$.   If $\alpha$ is an odd number less than  or equal to $m/t$,  then by setting $S_{\alpha}= \{a^m\} \cup \left(\bigcup\limits _{j=0}^{(\alpha-3)/2} \{a^{t(2j+1)}, a^{-t(2j+1)}\} \right)$, we get  that  $H$ is  an $(\alpha, 0)$-regular set in  CayS$(G,S_{\alpha})$. Similarly,  if $\alpha$ is an  even number less than $m/t$, then set $S_{\alpha}= \bigcup\limits_{j=0}^{(\alpha-2)/2} \{a^{t(2j+1)}, a^{-t(2j+1)}\} $.  Now, let  $S=S_{\alpha}\cup \Omega_{\beta}$.  If $S\not =\emptyset$, then  $H$ is an $(\alpha, \beta)$-regular set in ${\rm CayS}(G,S)$.

%Note that, as  $S\subseteq \{a^m\}\cup \left(\bigcup\limits_{i=0}^{(m-3)/2} \{a^{(2i+1)}, a^{-(2i+1)}\} \right)=\Omega_{m/t}\cup \{a^m\} \cup  \left(\bigcup\limits _{i=0}^{((m/t)-3)/2} \{a^{t(2i+1)}, a^{-t(2i+1)}\} \right)$, so $\alpha>m/t$ or $\beta>m/t$  for  described $S$ in this case does not occur. 

$\bullet$
 Let $b^G \subseteq S$ and $ S \cap (ab)^G = \emptyset$. Therefore,  $b^G\subseteq S\subseteq T= b^G\cup \{a^m\}\cup \left(\bigcup\limits_{i=0}^{(m-3)/2} \{a^{(2i+1)}, a^{-(2i+1)}\} \right)=b^G\cup \{a^m\}\cup \Omega_{m/t}\cup  \left(\bigcup\limits_{j=0}^{((m/t)-3)/2} \{a^{t(2j+1)}, a^{-t(2j+1)}\} \right)$. By Claim 1, $m/t=| b^G\cap H| \leq | S\cap H|=\alpha\leq \left|\left( b^G\cup \{a^m\}\cup \left(\bigcup\limits _{j=0}^{((m/t)-3)/2} \{a^{t(2j+1)}, a^{-t(2j+1)}\} \right)\right) \cap H\right|=2m/t $ and by Claims  2 and 3,  $m/t=| b^G \cap Hx| \leq |S \cap Hx|= \beta\leq |(\Omega_{m/t}\cup b^G)\cap Hx|=2m/t$, for $x \in G \setminus H$. 

Now,  for each $ m/t \leq \alpha, \beta  \leq 2m/t$, we  give a normal square-free subset $S$  of $G$ (with given  condition)  such that $H$ is  an $(\alpha,\beta)$-regular set in ${\rm CayS}(G, S)$.  If $\alpha -(m/t)$ is odd, then let 

$$S=\left(\bigcup\limits_{j=0}^{(\alpha-(m/t)-3)/2} \{a^{t(2j+1)}, a^{-t(2j+1)}\}\right)\cup \{a^m\}\cup b^G\cup \Omega_{\beta-(m/t)}.$$ 
 If $\alpha -(m/t)$ is even, then let 
$$S=\left(\bigcup\limits_{j=0}^{(\alpha-(m/t)-2)/2} \{a^{t(2j+1)}, a^{-t(2j+1)}\}\right)\cup b^{G} \cup \Omega_{\beta-(m/t)} .$$
 By Claims 1,2,3, we have  $| S\cap H|=\alpha$ and $|S \cap Hx|=\beta$, for each $x\in G\setminus H$. 

%B y Claim 3, for each $ m/t \leq \beta \leq 2m/t$, there exists a normal square-free subset $\Omega_{\beta - m/t}$ which is the union of $\beta - m/t$  sets $\mathfrak{S}_i$. 
% and also if $S$ is a normal square-free subset with the hypothesis of this case, then $H$ is an  $(\alpha, \beta)$-regular set in ${\rm CayS}(G,S)$ for some $m/t\leq \alpha,\beta \leq 2m/t $.

$\bullet$
If $(ab)^G \subseteq S$ and $S \cap b^G = \emptyset$, we have the same possibilities  for $\alpha$ and $\beta$ as the above case.  

$\bullet$
Suppose $b^G\cup (ab)^G\subseteq S$. Thus,  

\begin{eqnarray*}
b^G\cup (ab)^G\subseteq S\subseteq T=b^G\cup (ab)^G\cup  \{a^m\}\cup \left(\bigcup_{i=0}^{(m-3)/2} \{a^{(2i+1)}, a^{-(2i+1)}\} \right) \\
=  b^G\cup (ab)^G\cup \{a^m\}\cup \Omega_{m/t}\cup  \left(\bigcup_{j=0}^{((m/t)-3)/2} \{a^{t(2j+1)}, a^{-t(2j+1)}\} \right).
\end{eqnarray*}
Similarly to the above, by  Claims 1,2,3,   we have $2m/t\leq \alpha , \beta\leq 3m/t$. Now,  for each  $2m/t \leq \alpha , \beta \leq 3m/t$, we give   a normal square-free subset $S$ of $G$ such that $H$ is an  $(\alpha,\beta)$-regular set  in ${\rm CayS}(G,S)$. If $\alpha -(2m/t)$ is odd, then let $S$ be the union of $(\alpha-(2m/t)-1)/2$ sets of type $\{a^{t(2j+1)},a^{-t(2j+1)}\}$ with $\Omega_{\beta-(2m/t)} \cup b^{G} \cup (ab)^{G} \cup \{a^{m}\}$. If $\alpha -(2m/t)$ is even, then let $S$ be the union of $(\alpha-(2m/t))/2$ sets of type $\{a^{-t(2j+1)},a^{t(2j+1)}\}$ with $\Omega_{\beta-(2m/t)} \cup b^{G} \cup (ab)^{G}$, as wanted. 

Therefore, we see that  in the case that $t$ and $m$ are odd, all the possibilities for $(\alpha, \beta)$ are as described in Part (4) of the statement of the theorem.

{\bf Subcase 2.2} Let $m$ be even. Similarly to the above we consider the following cases for $S$. 

Assume that $S \cap (b^{G} \cup (ab)^{G})= \emptyset $. Thus, $$S \subseteq \bigcup\limits_{i=0}^{m/2-1} \{a^{(2i+1)}, a^{-(2i+1)}\}=\Omega_{m/t}\cup \left( \bigcup\limits_{j=0}^{(m/(2t))-1} \{a^{t(2j+1)}, a^{-t(2j+1)}\} \right),$$
As $H\cap \Omega_{m/t}=\emptyset$,  we deduce  that $0 \leq \alpha=|S\cap H|\leq m/t $ is even and $0\leq \beta=|S\cap Hx|\leq |\Omega_{m/t}\cap Hx|=m/t$, by Claim 3, for $x\in G\setminus H$.   Hence, setting $S= \Omega_\beta \cup \left( \bigcup\limits_{j=0}^{\alpha/2} \{a^{t(2j+1)}, a^{-t(2j+1)}\} \right)$, $H$ is an $(\alpha, \beta)$-regular set of $G$, for  all $0\leq \beta\leq m/t$ and each  even numbers $0\leq \alpha\leq m/t$.

 Now, assume that $S \cap (b^{G} \cup (ab)^{G}) \ne \emptyset $. First, suppose that $S$ contains exactly one of them. Without loss of generality, suppose that $b^{G} \subseteq S$.
Therefore,  $b^G\subseteq S\subseteq T= b^G\cup  \left(\bigcup\limits_{i=0}^{(m-2)/2} \{a^{(2i+1)}, a^{-(2i+1)}\} \right)=b^G\cup \Omega_{m/t}\cup  \left(\bigcup\limits_{j=0}^{((m/t)-2)/2} \{a^{t(2j+1)}, a^{-t(2j+1)}\} \right)$.
 By Claim 1, $m/t=| b^G\cap H| \leq | S\cap H|=\alpha\leq \left|  \left( b^G\cup  \left(\bigcup\limits_{j=0}^{((m/t)-2)/2} \{a^{t(2j+1)}, a^{-t(2j+1)}\} \right)\right)\cap H \right|=2m/t $   and $\alpha$ is even.  Also,  by Claims  2 and 3, $m/t=| b^G \cap Hx| \leq |S \cap Hx|= \beta\leq |(b^G\cup \Omega_{m/t} )\cap Hx|=2m/t$, for $x \in G \setminus H$.

%Let  $ m/t \leq \alpha, \beta  \leq 2m/t$. Now,  we give a normal square-free subset $S$  of $G$ (with given  condition)  such that $H$ is  an $(\alpha,\beta)$-regular set in ${\rm CayS}(G, S)$.  If $\alpha -(m/t)$ is odd, then let 
%$$S=\left(\bigcup\limits_{i=0}^{(\alpha-(m/t)-3)/2} \{a^{t(2i+1)}, a^{-t(2i+1)}\}\right)\cup \{a^m\}\cup b^G\cup \Omega_{\beta-(m/t)}.$$ 
 %If $\alpha -(m/t)$ is even, then let 
%$$S=\left(\bigcup\limits_{i=0}^{(\alpha-(m/t)-2)/2} \{a^{t(2i+1)}, a^{-t(2i+1)}\}\right)\cup b^{G} \cup \Omega_{\beta-(m/t)} .$$
 %Obviously, $| S\cap H|=\alpha$ and $|S \cap Hx|=\beta$, for each $x\in G\setminus H$. 

Now,  for each $ m/t \leq \alpha, \beta  \leq 2m/t$, where $\alpha$ is even, we  give a normal square-free subset $S$  of $G$ (with given  condition)  such that $H$ is  an $(\alpha,\beta)$-regular set in ${\rm CayS}(G, S)$.  
 Note that, as $\alpha$ is even in this case, we have  $\alpha-(m/t)$  is  even. 
Let $S$  be a union of $(\alpha-(m/t))/2$   conjugacy classes of type $\{a^{t(2j+1)}, a^{-t(2j+1)}\}$ with $b^G\cup \Omega_{\beta-(m/t)}$.  Hence, $H$ is an $(\alpha, \beta)$-regular  set of $G $.   For the case that $(ab)^G\subseteq S$ and $b^G\cap S=\emptyset$, we have the same possibilities for $(\alpha, \beta)$.   

Let  $b^G\cup (ab)^G\subseteq S$.  Thus

\begin{eqnarray*}
b^G\cup (ab)^G\subseteq S\subseteq T=b^G\cup (ab)^G\cup  \{a^m\}\cup \left(\bigcup_{i=0}^{(m-2)/2} \{a^{(2i+1)}, a^{-(2i+1)}\} \right) \\
=  b^G\cup (ab)^G\cup \{a^m\}\cup \Omega_{m/t}\cup  \left(\bigcup_{j=0}^{((m/t)-2)/2} \{a^{t(2j+1)}, a^{-t(2j+1)}\} \right).
\end{eqnarray*}

Similarly to the above, by  Claims 1,2,3,   we have $2m/t\leq \alpha , \beta\leq 3m/t$, where $\alpha$ is even. Now,  for each  $2m/t \leq \alpha , \beta \leq 3m/t$,  where $\alpha$ is even, we give   a normal square-free subset $S$ of $G$ such that $H$ is an  $(\alpha,\beta)$-regular set  in ${\rm CayS}(G,S)$.  Noting  that  $\alpha -(2m/t)$ is even,  let $S$ be the union of $(\alpha-(2m/t))/2$  conjugacy classes of type $\{a^{-t(2j+1)},a^{t(2j+1)}\}$ with $\Omega_{\beta-(2m/t)} \cup b^{G} \cup (ab)^{G}$, as wanted. 

Therefore, we see that  in the case that $t$  is odd and $m$ is even, all the possibilities for $(\alpha, \beta)$ are as described in Part (5) of the statement of the theorem. 
% then $H$ is an $(\alpha, \beta)$-regular set of $G$,  for $ 2m/t \leq \alpha,\beta\leq (3m/t)$, where $\alpha$ is even, with the same discussion as before. 
%\end{enumerate}
 %\end{enumerate}
\end{proof}

\begin{theorem}
Let  $n\geq 3$ be odd and  $t$ be a divisor of $n$. Let   $H=\left\langle a^{t} \right\rangle$ be a subgroup of $G=D_{2n}$.  Then $H$ is an $(\alpha, \beta)$-regular set  of $G$  if and only if   $t=1$ and $(\alpha, \beta)=(0, |H|)$.

\end{theorem}
\begin{proof}
 %For getting the result, similarly to the Theorem \ref{3.5}, first we assume that $H$ is an   $(\alpha, \beta)$-regular set   in  {\rm CayS}$(G,S)$,  for some  subset  $\emptyset\not =S\subseteq G$ and  we find some restrictions on   $(\alpha, \beta)$. Then for each $(\alpha, \beta)$ satisfying those restrictions we give a normal square-free subset $S$ such that   $H$ is an   $(\alpha, \beta)$-regular set   in  {\rm CayS}$(G,S)$.
Noting that NSq$(G)=b^G$ and,  as  the connection set $S\not = \emptyset$, we get that $S=b^G$. So $\alpha=|S\cap H|=0$.   If $t>1$, then $\beta\leq\mathcal{L}(H)\leq  |$NSq$(G)\cap Ha|=0$, which is excluded.  Thus, $t=1$ and $\beta=|H|/2$, as desired.  
\end{proof}

\begin{theorem}
Let     $n = 2m$, for some positive integer $m \geq 2$ and $t$ be a  divisor of $n$.  Let $H=\left\langle a^{t} \right\rangle$ be a subgroup of $G=D_{2n}$. Then $H$ is an $(\alpha, \beta)$-regular set of  $G$, where $(\alpha, \beta)\not=(0,0)$,   if and only if one of the following occurs:

1)   $(\alpha, \beta)= (0,|H|)$,  when  $t=2$.

2) $0\leq \alpha \leq  |H|/2$ and $\beta\in \{0, |H|/2, |H|\}$, when $t=1$ and $m$ is odd.

3) $\alpha=2\gamma $,  for some $0\leq \gamma\leq  |H|/4$,  and $\beta\in \{0, |H|/2, |H|\}$, when $t=1$ and $m$ is even.

4)  $0\leq \alpha \leq  |H|/2$ and $\beta\in \{0, |H|/2\}$,  when $t>1$   and $m$ are   odd.

5)  $\alpha=2\gamma $,  for some $0\leq \gamma\leq  |H|/4$,  and $\beta\in \{0, |H|/2\}$, when $t> 1$ is odd and $m$ is even.
\end{theorem}
\begin{proof} 

Note that in this case $H$ is a  normal subgroup of $G$ and so by  the proof of Lemma \ref{normal}, $H$ is an $(\alpha, \beta)$-regular set in  ${\rm CayS}(G, S)$, for some subset $\emptyset\not =S\subseteq G$ if and only if  $H$ is an $(\alpha, 0)$-regular set  in   ${\rm CayS}(G, S\cap H)$ and $H$ is a $(0, \beta)$-regular set   in  ${\rm CayS}(G, S\cap (G\setminus H))$. 

 For getting the result, first we assume that $H$ is an   $(\alpha, \beta)$-regular set   in  {\rm CayS}$(G,S)$,  for some  subset  $\emptyset\not =S\subseteq G$ and  we find some restrictions on   $(\alpha, \beta)$. Then for each $(\alpha, \beta)$ satisfying those restrictions we give a normal square-free subset $S$ such that   $H$ is an   $(\alpha, \beta)$-regular set   in  {\rm CayS}$(G,S)$. 
We consider the following two cases: 
	 
{\bf Case 1.} Assume that $t$ is an even divisor of $n$. Thus, $H\subseteq {\rm Sq}(G)$  and  hence $\alpha=0$. Thus,  as $S\not =\emptyset$,  we must have $\beta\geq 1$. 
%\begin{enumerate}
%		\item 

First, assume $t>2$. In this case, $a^2 \notin H$ and $Ha^2 \subseteq {\rm Sq}(G)$. So, by Theorem \ref{first},  $\beta\leq \mathcal{L}(H)\leq |$NSq$(G)\cap Ha^2|=0$. Therefore,  for even integer $t>2$, $H$ is not an  $(\alpha,\beta)$-regular set of $G$.   

Now, let $t=2$. Then $H=$ Sq$(G)$ and $[G : H]=4$. Therefore, $G\setminus H= Ha\cup Hb \cup Hab$, where $Hb=b^G$ and $Hab=(ab)^G$.  If $m$ is even, then $Ha=\{a^{2j+1},a^{-(2j+1)}: 0 \leq j \leq m-1\}$, and in case $m$ is odd,  $Ha=\{a^{2j+1},a^{-(2j+1)}: 0 \leq j \leq m-1\} \cup \{a^m\}$.  As $\beta\geq 1$,  we have  $S\cap Hb\not =\emptyset$ and $S\cap Hab\not =\emptyset$. Therefore, by the normality of $S$ we conclude that $Hab\cup Hb\subseteq S$ and since $|S \cap Hg|=\beta$, for each $g \in G \setminus H$, we get that $S=G\setminus H $ and so $\beta=|H|$. It follows that  the only possibility is $\beta=|H|$, as described in Part (1) of the theorem.

%\end{enumerate}

{\bf Case 2.} Assume that $t$ is an odd divisor of $n$. In this case, ${\rm NSq}(G)\cap H=  H\setminus \langle a^{2t}\rangle$. Then, $0 \leq \alpha=|S \cap H| \leq |{\rm NSq}(G)\cap H|=|H|/2$. 

 In the following  we discuss  the possibilities for $\alpha$ and for each  feasible $\alpha$ we  introduce a normal square-free  subset  of $H$ of size $\alpha$, say  $S_{\alpha}$.

$\bullet$
Let $m$ be odd. If $1\leq \alpha\leq |H|/2$ is odd, then  we set  $S_{\alpha}$  to be a union of $\{a^m\}$ with $(\alpha-1)/2$ sets of type $\{a^{kt}, a^{-kt}\}$, where   $k< m/t$ is odd.   If   $0\leq \alpha\leq |H|/2$ is even, then   we set $S_{\alpha}$ to be a union of  $\alpha/2$ sets of type $\{a^{kt}, a^{-kt}\}$,  where   $k< m/t$ is odd.

$\bullet $ Let   $m$ be even. Then $\alpha$ is even, as $S \cap H\subseteq \bigcup\limits_{i=0}^{((m/t)-2)/2} \{a^{t(2i+1)}, a^{-t(2i+1)}\}$,   a union of some square-free conjugacy classes of size $2$. Let $S_{\alpha}$  be a union of  $\alpha/2$ sets of type $\{a^{kt}, a^{-kt}\}$,  where $k< m/t$ is odd.

Therefore $H$ is an $(\alpha,0)$-regular set in ${\rm CayS}(G,S_{\alpha})$. Now, in the  following  two cases we discuss  possibilities for  $\beta$  and in each case we give normal square-free subsets $S$ such that $H$ is an $(\alpha, \beta)$-regular set in ${\rm CayS}(G,S)$.

$\blacktriangleright$ Let $t=1$.

  Then $[G : H]=2$. Therefore, $G\setminus H=Hb=b^G \cup (ab)^G$. By the normality of $S$ we conclude that $S\cap (G\setminus H)$ can be equal to either $\emptyset$,  $b^G$,  $(ab)^G$ or $b^G \cup (ab)^G$. Hence, $\beta=|S\cap (G\setminus H)|\in \{0, |H|/2, |H|\}$. 

Thus, setting $S=S_{\alpha}$ (if $\alpha\geq 1$ and $\beta=0$), $S=S_{\alpha}\cup b^G$ or $S_{\alpha}\cup b^G\cup (ab)^G$   we have $H$ is an $(\alpha, \beta)$-regular set in ${\rm CayS}(G,S)$, where $0\leq \alpha\leq |H|/2$ and $\beta\in \{0, |H|/2, |H|\}$. Note that by previous discussion $\alpha$ must be even, when $m$ is even. Thus, we have  Parts (2) and (3) of the theorem.

%In the following, we study  the possiblities for $\beta$, when $t>1$. 

$\blacktriangleright$  Let $t>1$.

Then,  ${\rm NSq}(G)\cap  Ha^ib=Ha^ib\subseteq b^G\cup (ab)^G$, $|(ab)^G \cap Ha^ib|=| b^G \cap Ha^ib|=|H|/2$  and $|$NSq$(G)\cap Ha^{2i}|=|$NSq$(G)\cap Ha^{2i+1}|=|H|/2$, for every integer $i$. Hence, if  $\beta\geq 1$, then $ b^G \cap Hb\subseteq S$ or $ (ab)^G \cap Hb\subseteq  S$, implying that $\beta\geq |H|/2$ and as  $|$NSq$(G)\cap Ha|=|H|/2$,  we get that $\beta\leq \mathcal{L}(H)= |H|/2$, yields   to  $\beta=|H|/2$.   Thus, $\beta\in \{0, |H|/2\}$.  

 Let $\Omega=\{a^k\ : \ \text{where $k$ is an  odd  integer not divided by $t$}\}\subseteq {\rm NSq}(G)$. Then, setting $S=S_{\alpha}$ (if $\alpha\geq 1$ and $\beta=0$) we have $H$ is an $(\alpha, 0 )$-regular set in ${\rm CayS}(G, S)$.   Let  $S=b^G\cup \Omega \cup S_{\alpha}$. Then, as $|Ha^i\cap S|=|Ha^i \cap \Omega |=|H|/2$, for $1\leq i\leq t-1$,   and $|Ha^ib\cap S|=|Ha^ib\cap b^G|=|H|/2$, we have  $H$ is an $(\alpha, |H|/2)$-regular set in ${\rm CayS}(G, S)$, where $0\leq \alpha\leq |H|/2$.  Remind  that,    $\alpha$ is even if $m$ is even. So,  we have Parts  (4) and (5) of the theorem. 
%\end{enumerate}
\end{proof}

{\bf Acknowledgement.} The authors gratefully thank  the referees  for the constructive comments and recommendations
which definitely help to improve the quality and readability  of the paper.

\end{document}